\documentclass[12pt]{article}
\usepackage{latexsym, amsmath, amsfonts, amsthm, amssymb}
\usepackage{times}
\usepackage{a4wide}
\usepackage{graphicx, tocloft}
\usepackage[hyphens]{url}
\usepackage{hyperref}
\usepackage[hyphenbreaks]{breakurl}

\newcommand{\cov}{\mathrm{Cov}}

\newcommand{\var}{\mathrm{Var}}
\newcommand{\tr}{\mathrm{Tr}}
\newcommand{\RR}{\mathbb{R}}

\newcommand{\EE}{\mathbb{E}}

\newcommand{\PP}{\mathbb{P}}

\newcommand{\eps}{\varepsilon}
\newcommand{\vphi}{\varphi}
\newcommand{\id}{\mathrm{Id}}

\newtheorem{theorem}{Theorem}[section]

\newtheorem{lemma}[theorem]{Lemma}

\newtheorem{proposition}[theorem]{Proposition}
\newtheorem{corollary}[theorem]{Corollary}

\theoremstyle{remark}
\newtheorem{remark}[theorem]{Remark}

\begin{document}

\title{Bourgain's slicing problem and KLS isoperimetry up to polylog}

\author{
Bo'az Klartag\footnote{
Department of Mathematics,
Weizmann Institute of Science,
Rehovot 76100, Israel
}\; and Joseph Lehec\footnote{
CEREMADE (UMR CNRS 7534),
Universit\'e Paris-Dauphine, 75016 Paris, France}
}
\date{}
\maketitle

\abstract{We prove that Bourgain's hyperplane conjecture
and the Kannan-Lov\'asz-Simonovits (KLS) isoperimetric conjecture
hold true up to a factor that is polylogarithmic in the dimension.}
\section{Introduction}
One of the central questions in high-dimensional convex geometry is Bourgain's slicing problem.
In its simplest formulation, it asks whether for any convex body $K \subseteq \RR^n$ of volume one, there exists a hyperplane $H \subseteq \RR^n$ such that
$$
{\rm Vol}_{n-1}(K \cap H) > c,
$$
where $c > 0$ is a universal constant,
and ${\rm Vol}_{n-1}$ stands for $(n-1)$-dimensional volume. The question originated in Bourgain's work \cite{bou2, bou1}, and we refer the reader to \cite{KM} and references therein for background on this problem, its implications for convexity and its equivalent formulations. For $n \geq 1$ define
$$ \frac{1}{L_n} := \inf_{K \subseteq \RR^n} \sup_{H \subseteq \RR^n} {\rm Vol}_{n-1}(K \cap H), $$
where the infimum runs over all convex bodies $K \subseteq \RR^n$ of volume one, and the supremum  runs over all hyperplanes $H \subseteq \RR^n$. Thus, a convex body of volume one in $\RR^n$ has a hyperplane section whose $(n-1)$-dimensional volume is at least $1 / L_n$. For decades, the best estimate for $L_n$
has been the bound $L_n \leq C n^{1/4}$ where $C > 0$ is a universal constant, as proven in Bourgain \cite{bou3, bou4} (up to a logarithmic factor) and in \cite{K_quarter}. A recent breakthrough by Chen \cite{chen} has led to the bound $L_n \leq C_{\eps} n^{\eps}$ for any $\eps > 0$, or more precisely
\begin{equation}  L_n \leq C_1 \exp  \left(C_2 \sqrt{\log n} \cdot \sqrt{\log \log (3n) } \right), \label{eq_415}
\end{equation}
where $C_1, C_2 > 0$ are universal constants. Chen arrives at (\ref{eq_415}) by exploiting the relation between the slicing problem and the  thin shell problem, due to Eldan and Klartag \cite{EK}. It is proven in \cite{EK} that
\begin{equation}
L_n \leq C \sigma_n, \label{eq_435}
\end{equation}
where $C > 0$ is a universal constant, and where $\sigma_n$ is the thin-shell constant which we will describe shortly.
A probability density $\rho: \RR^n \rightarrow [0, \infty)$ is {\it log-concave} if the set $\{ \rho > 0 \} = \{ x \in \RR^n \, ; \, \rho(x) > 0 \}$ is convex, and $\log \rho$ is concave in $\{ \rho > 0 \}$. A probability measure in $\RR^n$ (or a random vector in $\RR^n$) is log-concave if it is supported
in an affine subspace of $\RR^n$ and it has a log-concave density in this subspace. For instance, the uniform probability measure on any compact, convex set is log-concave, as well as all Gaussian measures.  We say that a probability measure $\mu$ on $\RR^n$ with finite second moments is \emph{isotropic} if
\begin{equation}  \int_{\RR^n} x_i \, d \mu(x) = 0 \quad \text{and} \quad
\int_{\RR^n} x_i x_j \, d \mu(x) = \delta_{ij} \qquad  (i,j=1,\ldots,n), \label{eq_428} \end{equation}
where $\delta_{ij}$ is Kronecker's delta. Thus $\mu$ is isotropic when it is centered and its covariance matrix is the identity matrix. A log-concave probability measure has finite moments of all orders  (e.g. \cite[Lemma 2.2.1]{BGVV}).
The convolution of two log-concave
probability measures is again log-concave, as follows from the Pr\'ekopa-Leindler inequality \cite[Theorem 1.2.3]{BGVV} or from the earlier work
by Davidovi\v{c}, Korenbljum and Hacet \cite{DKH}. The relevance of the class of log-concave distributions
to the slicing problem was realized by Ball \cite{ball_studia}.
The thin-shell constant $\sigma_{\mu} > 0$ of an isotropic, log-concave probability measure $\mu$ in $\RR^n$ is defined via
\begin{equation}  n \sigma_{\mu}^2 =  \var_{\mu} (|x|^2), \label{eq_430} \end{equation}
where $\var_{\mu}(f) = \int f^2 d \mu - \left( \int f d \mu \right)^2$. It may be shown that most of the mass of
the measure $\mu$ is located in a  spherical shell whose width is at most $C \sigma_{\mu}$, and this estimate for the width is always tight, hence the name {\it thin-shell constant}, see
Anttila, Ball and Perissinaki \cite{abp} and Bobkov and Koldobsky \cite{BK1}. The thin-shell
constant  is
crucial for establishing the Central Limit Theorem for Convex Sets \cite{K_invent, ptrf},
as put forth in \cite{abp} following Sudakov \cite{sudakov} and Diaconis and Freedman \cite{DF}.
The parameter $\sigma_n$ mentioned above is defined as
$$ \sigma_n = \sup_{\mu } \sigma_{\mu} $$
where the supremum runs over all isotropic, log-concave probability measures $\mu$ in $\RR^n$. Earlier
bounds for $\sigma_n$ utilized the Concentration of Measure Phenomenon, see \cite{K_poly}, Fleury \cite{fleury} and Gu\'edon and Milman
\cite{GM}, following Paouris' large deviation principle \cite{paouris}.
More recent advances, due to Eldan \cite{Eldan1}, Lee and Vempala \cite{LV} and Chen \cite{chen}, deal with the Poincar\'e
constant. The Poincar\'e constant  $C_P(\mu)$ of a Borel probability measure $\mu$ in $\RR^n$ is defined as the smallest constant
$C \geq 0$ such that for any locally-Lipschitz function $f \in L^2(\mu)$,
\begin{equation}  \var_{\mu}(f) \leq C \cdot \int_{\RR^n} |\nabla f|^2 d \mu. \label{poincare} \end{equation}
The fact that $\sigma^2_{\mu} \leq 4 C_P(\mu)$ when $\mu$ is isotropic is easily proven:
\begin{equation} n \sigma_{\mu}^2 = \var_{\mu}(|x|^2)
\leq C_P(\mu)  \int_{\RR^n} |2x|^2 d \mu(x) = 4 n \cdot C_P(\mu).
\label{eq_1052} \end{equation}
The Poincar\'e constant is closely related to the {\it isoperimetric constant} or the {\it Cheeger constant} of $\mu$.
Given a probability measure $\mu$ in $\RR^n$ with log-concave density $\rho$, its isoperimetric constant is
$$ \frac{1}{\psi_{\mu}} = \inf_{A \subseteq \RR^n} \left\{ \frac{\int_{\partial A} \rho}{\min \{ \mu(A), 1 - \mu(A) \}} \right\} $$
where the infimum runs over all open sets $A \subseteq \RR^n$ with smooth boundary
for which $0 < \mu(A) < 1$.
By the Cheeger inequality \cite{cheeger} and the Buser-Ledoux inequality \cite{buser, ledoux},
for any absolutely-continuous, log-concave probability measure $\mu$ in $\RR^n$,
\begin{equation}
\frac{1}{4} \leq \frac{\psi_{\mu}^2}{C_P(\mu)} \leq 9,
\label{eq_536}
\end{equation}
where the inequality on the left, Cheeger's inequality, is rather general and does not require log-concavity. Define
\begin{equation}
\psi_n := \sup_{\mu} \psi_{\mu} \label{eq_1210} \end{equation}
where the supremum runs over all isotropic,
log-concave probability measures $\mu$ in $\RR^n$.
Our convention follows
Lee and Vempala \cite{LV}; the quantity
we denote by $\psi_n$ is denoted by $G_n$ in Eldan \cite{Eldan1} and by $\psi_n^{-1}$ in Chen \cite{chen}. The Kannan-Lov\'asz-Simonovits (KLS) conjecture \cite{KLS} suggests that $\psi_n$ is bounded by a universal constant.
Thanks to
(\ref{eq_435}), (\ref{eq_1052}) and (\ref{eq_536}) we have
the chain of inequalities
\begin{equation} L_n \leq C \sigma_n \leq \tilde{C} \psi_n, \label{eq_550} \end{equation}
where $C, \tilde{C} > 0$ are universal constants.
The right-hand side inequality in (\ref{eq_550}) may be reversed, up to a logarithmic factor.
A deep theorem by Eldan \cite{Eldan1} indeed states that
\begin{equation}
\psi_n \leq \tilde{C} \sigma_n \cdot \log n.
\label{eq_1832}
\end{equation}
In \cite{chen}, Chen uses {\it Eldan's stochastic localization}
\cite{Eldan1}  and the analysis of Lee and Vempala
\cite{LV} in order to show that
\begin{equation} \psi_n \leq C_1 \exp  \left(C_2 \sqrt{\log n} \cdot \sqrt{\log \log (3n) } \right), \label{eq_546} \end{equation}
where $C_1, C_2 > 0$ are universal constants. This bound implies (\ref{eq_415}), in view of (\ref{eq_550}).
The bound in (\ref{eq_546}) grows slower than any power law, and it is natural to expect that this bound for $\psi_n$ may be improved to a polylogarithmic one. This is indeed true, as we show in this paper:
\begin{theorem} For any $n \geq 2$,
$$ \psi_n \leq C (\log n)^{\alpha} $$
for some universal constants $C, \alpha > 0$.
Our proof yields  $\alpha = 5 $.
\label{thm_2226}
\end{theorem}
More precisely, we actually prove that $\sigma_n \leq C (\log n)^4 $.
This implies Theorem \ref{thm_2226}, according to \eqref{eq_1832}. Moreover, thanks to (\ref{eq_435}) we conclude from our bound for $\sigma_n$  that
\begin{equation} L_n\leq C' (\log n)^4 . \label{eq_1528} \end{equation}
As in Lee-Vempala \cite{LV} and Chen \cite{chen},
our argument relies on Eldan's stochastic localization, with the new ingredients
being the functional-analytic approach from Klartag and Putterman \cite{KP}, as well as an $H^{-1}$-inequality from \cite{ptrf}.
In our proof we analyze the evolution of a log-concave
measure $\mu$ along the heat flow $(P_s)_{s \geq 0}$ in $\RR^n$ using the
functional-analytic formalism from \cite{KP}. Eldan's stochastic localization enters
the picture through the time reversal $t:= 1/s$, and it allows us to exploit the isotropicity of $\mu$.

\medskip Throughout this note, the letters $c, C, \tilde{c}, C_1, c_2$ etc. denote positive universal constants, whose value may change from one line to the next. We usually use lower-case $c, \tilde{c}$ to denote universal constants that we view as sufficiently small, while $C, \hat{C}$ etc. usually denote constants which we view as sufficiently large. Since this paper
is concerned with asymptotics in the dimension $n$,
we may assume in the proof that $n$ exceeds a certain universal constant.
We write $x \cdot y = \langle x, y \rangle = \sum_i x_i y_i$ for the standard scalar
product between $x,y \in \RR^n$ and $|x| = \sqrt{\langle x, x \rangle}$. For symmetric matrices $A, B \in \RR^{n \times n}$
we write $A \leq B$ if $B - A$ is positive semi-definite.
By a smooth function we mean $C^{\infty}$-smooth,
and $\log$ is the natural logarithm.

\medskip {\it Acknowledgements.} The first-named author would like to thank Ronen Eldan and Eli Putterman for interesting discussions, and was partially supported by a grant from the Israel Science Foundation (ISF).

\section{Log-concave measures along the heat flow}

Let $\mu$ be an isotropic, log-concave probability measure  in $\RR^n$ with a smooth, positive density.
We begin by recalling the setting of Klartag and Putterman \cite{KP} as well as the basic properties of the Laplace 
operator that is associated with $\mu$. For $s > 0$, we write $\gamma_s$ for the density
of a Gaussian random vector of mean zero and covariance $s \cdot \id$ in $\RR^n$. Denote
$$ \mu_s = \mu * \gamma_s, $$
the convolution of $\mu$ and $\gamma_s$, with $\mu_0 = \mu$.
The heat operator $P_s f = f * \gamma_s$ is a contraction from $L^2(\mu_s)$
to $L^2(\mu)$. The adjoint operator $Q_s = P_s^*: L^2(\mu) \rightarrow L^2(\mu_s)$
satisfies
\begin{equation}  Q_s \vphi = \frac{ P_s ( \vphi \rho ) }{P_s \rho} \label{eq_2017} \end{equation}
where $\rho$ is the log-concave density of $\mu$. We define $Q_s \vphi$ via formula (\ref{eq_2017})
for all $s > 0$ and $\vphi \in L^1(\mu)$. We set $P_0 = \id$ and $Q_0 = \id$.
The Laplace operator associated with $\mu$ is the operator $L = L_{\mu}$, initially defined for compactly-supported smooth functions
via the formula
$$ L u = \Delta u + \nabla (\log \rho) \cdot \nabla u. $$
By integration by parts, it follows that for any two smooth functions $u,v: \RR^n \rightarrow \RR$,
if one of them is compactly-supported then
\begin{equation}  \int_{\RR^n} (Lu) v d\mu = -\int_{\RR^n} (\nabla u \cdot \nabla v) d \mu. \label{eq_1101} \end{equation}
It is proven e.g. in \cite[Corollary 3.2.2]{BGL} that $L$ is essentially self-adjoint in $L^2(\mu)$ and negative semi-definite.
We may thus extend the domain of definition of $L$,
and from now on we denote by $L$ the closure in $L^2(\mu)$ of the operator previously denoted by $L$.
Note that the self-adjoint operator $L$ has a simple
eigenvalue at $0$ corresponding to the constant eigenfunction.
We refer the reader to \cite{BGL, D, EMT} for  standard
background on spectral theory. By the spectral theorem (e.g., \cite[Theorem A.4.2]{BGL} or \cite[Theorem 11.5.1]{EMT}) we may write
\begin{equation}\label{eq_spectraldec} -L = \int_{-\infty}^{\infty} \lambda d E_{\lambda}
\end{equation}
for a certain increasing, right-continuous family of orthogonal projections $(E_{\lambda})_{\lambda \in \RR}$
with $\lim_{\lambda \rightarrow \infty} E_{\lambda} = \id$ and
$\lim_{\lambda \rightarrow -\infty} E_{\lambda} = 0$ in
the sense of strong convergence of operators. In particular,
\begin{equation}\label{eq_E0} E_0 f(x) = \int_{\RR^n} f d \mu
\qquad \qquad \text{for all} \ x \in \RR^n. \end{equation}
For  $f \in L^2(\mu)$
we write $\nu_f$ for the Borel measure on $\RR$ that satisfies
$\nu_f((a,b]) = \langle E_b f, f \rangle - \langle E_a f, f \rangle$ for all $a < b$, i.e.,
the spectral measure of $f$. Thus $$ \nu_f(\RR) = \| f \|_{L^2(\mu)}^2. $$
The Poincar\'e constant $C_P(\mu)$ of a log-concave
probability measure in $\RR^n$ is always finite,
by (\ref{eq_536}) and by Bobkov \cite{bobkov} or Kannan, Lov\'asz and Simonovits \cite{KLS}.
By the definition of the Poincar\'e constant,
\[
\lambda_1 := \frac 1{ C_P(\mu) }
\]
is the \emph{spectral gap} of $L$, in the sense that $E_\lambda = E_0$ for $\lambda <\lambda_1$. In other words, if $f\in L^2(\mu)$ satisfies $\int f d \mu = 0$  then 
\begin{equation}  \nu_f([0, \lambda_1)) = 0. \label{eq_1851} \end{equation}
The following proposition provides an upper bound for the spectral mass of
a given function $f$ below a certain level in terms of the $L^2 (\mu_s)$-norm of $Q_s f$.
\begin{proposition}\label{prop_1715}
Let $\mu$ be a probability measure with a smooth, positive density in $\RR^n$. 
Let $f\in  L^2 (\mu)$ satisfy $\int_{\RR^n} f \, d\mu =0$
and $\Vert f \Vert_{L^2(\mu)}= 1$.
Then for $s, \lambda > 0$, 
\[
\langle E_\lambda f , f\rangle_{L^2(\mu)}  \leq C \left(  \Vert Q_s f \Vert_{L^2(\mu_s)} + s \lambda \right),
\]
where $C > 0$ is a universal constant. Our proof gives $C = 4$.
\end{proposition}
We prove Proposition \ref{prop_1715} in Section \ref{sec_1700}. 
In order to estimate $\Vert Q_s f\Vert^2_{L^2(\mu_s)}$ from above it is convenient to set $t := 1/s$ and to use the framework  of Eldan's stochastic localization.
We refer the reader to \cite{chen, Eldan1, LV} for background on this subject
and for further explanations,
and we use the notation from \cite[Section 4]{KP}. For $t \geq 0$ and $\theta \in \RR^n$ consider
the probability density
$$ p_{t, \theta}(x) = \frac{1}{Z(t,\theta)} e^{\langle \theta, x \rangle - t |x|^2/2} \rho(x) \qquad \qquad \qquad (x \in \RR^n) $$
where $Z(t, \theta) = \int_{\RR^n} e^{\langle \theta, x \rangle - t |x|^2/2} \rho(x) dx$ and we recall that $\rho$
is the log-concave density of the measure $\mu$. The barycenter of $p_{t, \theta}$ is
$$ a(t,\theta) = \int_{\RR^n} x p_{t, \theta}(x) \, dx \in \RR^n. $$
We consider the ``tilt process'', the stochastic process $(\theta_t)_{t \geq 0}$ in $\RR^n$ that satisfies the stochastic differential equation
\begin{equation}  \theta_0 = 0 , \quad d \theta_t = d W_t + a(t, \theta_t) dt, \label{eq_910} \end{equation}
where $(W_t)_{t \geq 0}$ is a standard Brownian motion in $\RR^n$. The existence and uniqueness of a strong solution to (\ref{eq_910}) are standard.
Setting $p_t(x) = p_{t, \theta_t}(x)$ and $a_t = a(t, \theta_t)$ we obtain the equation of {\it Eldan's stochastic localization}
in the Lee-Vempala formulation:
\begin{equation}  d p_t(x) = p_t(x) \langle x - a_t, d W_t \rangle \qquad \qquad \qquad (x \in \RR^n)
\label{eq_1048} \end{equation}
with $p_0(x) = \rho(x)$.
\begin{lemma}\label{lem_1043}
Let $\vphi \in L^1(\mu)$ and $s > 0$. Consider the stochastic process $M_t = \int_{\RR^n} \vphi p_t$ defined for $t \geq 0$.
Then with $t = 1/s$,
$$ \EE M_t^2 = \| Q_s \vphi \|_{L^2(\mu_s)}^2. $$
\end{lemma}
\begin{proof} Let $y \in \RR^n$ and set $\theta = t y$. It follows from \cite[Lemma 2.1]{KP} that
\begin{equation}  Q_s \vphi (y) = \int_{\RR^n} \vphi p_{t, \theta} =: M(t, \theta). \label{eq_1011} \end{equation}
Moreover, the law of the random vector $\theta_t / t$ is the probability measure $\mu_s$, as explained in \cite[Section 4]{KP}. Therefore, by (\ref{eq_1011}),
$$ \EE M_t^2 = \EE \left[ M(t, \theta_t)^2 \right]= \EE \left[ (Q_s \vphi)^2 \left( \frac{ \theta_t}{t} \right) \right]
= \int_{\RR^n} (Q_s \vphi)^2 d \mu_s, $$
completing the proof.
\end{proof}
For a vector-valued function $f = (f_1,\ldots,f_n): \RR^n \rightarrow \RR^n$ we
define $Q_s f: \RR^n \rightarrow \RR^n$ coordinate-wise,
and denote $\| f \|_{L^2(\mu)}^2 = \sum_i \| f_i \|_{L^2(\mu)}^2$.
Since $$ a_t = \int_{\RR^n} x p_t(x) \, dx, $$
we conclude from Lemma \ref{lem_1043}
that for any $s > 0$, with $t = 1/s$,
\begin{equation}
\EE |a_t|^2 = \| Q_s x \|_{L^2(\mu_s)}^2. \label{eq_1044}
\end{equation}
For $t > 0$ and $\theta \in \RR^n$ denote
$$ A(t,\theta) = \cov(p_{t, \theta}) = \int_{\RR^n} \left[ x \otimes x \right] p_{t, \theta}(x) dx - a(t,\theta) \otimes a(t,\theta) \in \RR^{n \times n}, $$
the covariance matrix of $p_{t, \theta}$. Here $x \otimes x = (x_i x_j)_{i,j=1,\ldots,n} \in \RR^{n \times n}$
for $x = (x_1,\ldots,x_n) \in \RR^n$. Set $A_t = A(t, \theta_t)$. It follows from (\ref{eq_1048}) that
$$ d a_t = d \left[ \int_{\RR^n} x p_t(x) dx \right] = A_t d W_t $$
and consequently,
\begin{equation}  \frac{d}{dt} \EE |a_t|^2 = \EE \| A_t \|_2^2,
\label{eq_1049} \end{equation}
where we write $\| A_t \|_q = \tr[A_t^q]^{1/q}$ for the $q$-Schatten norm of the matrix $A_t$. Note that the random matrix $A_t$
is symmetric and positive definite, since it is the covariance matrix of an absolutely-continuous probability measure in $\RR^n$.

\medskip 
In addition to the parameters $\sigma_n$ and $\psi_n$ mentioned above, we shall also need the quantity $\kappa_n > 0$ defined via
\begin{equation}\label{eq_kappan}
\kappa_n^2 := \sup_X \sup_{\theta\in S^{n-1}}
\left\{  \left\Vert \EE \langle X,\theta\rangle (X\otimes X) \right\Vert_2^2 \right\} ,
\end{equation}
where the first supremum runs over all isotropic log-concave random vectors $X$ in $\RR^n$,
and where $S^{n-1} = \{ x \in \RR^n \, ; \, |x| = 1 \}$ is the unit sphere.
The relation between $\psi_n, \kappa_n$ and $\sigma_n$ proven by Eldan \cite{Eldan1} is
\begin{equation}\label{eq_eldanbis}
\psi_n^2 \leq  C \log n \cdot \kappa_n^2 \leq \tilde{C} \log^2 n \cdot \sigma_n^2,
\end{equation}
where $C, \tilde{C} > 0$ are universal constants. The core of Eldan's argument is the first inequality of~\eqref{eq_eldanbis},
whereas the second inequality  is relatively easy. The following lemma essentially follows from Eldan's ideas ~\cite{Eldan1}. However,
the stochastic localization used in \cite{Eldan1}
is slightly different, and the focus
there is on tail probabilities
rather than on the expectation of $\| A_t \|_q^q$. We thus provide a
detailed proof of the following lemma in Section~\ref{sec_eldan}.

\begin{lemma}\label{lem_eldan}
For $t\leq (C \kappa_n^2 \cdot \log n)^{-1}$ and $q\geq 1$ we have
\[
\EE \Vert A_t \Vert_q^q  \leq C_q n ,
\]
where $C > 0$ is a universal constant and $C_q$ depends only on $q$.
\label{lem_1135}
\end{lemma}
Applying this lemma with $q=2$ we see that $\EE \Vert A_t\Vert_2^2$ has  the order
of magnitude of $n = \| A_0 \|_2^2$  for all values of $t$ up to time  $t_1 = (C \kappa_n^2 \cdot \log n)^{-1}$.
We would like to show that this quantity cannot increase too rapidly also  beyond time $t_1$. This is the content of the next lemma,
which is rather similar to Chen's lemma
\cite[Lemma 8]{chen}.   Chen's lemma
is the key ingredient of his sub-polynomial bound for the KLS constant.

\begin{lemma}\label{lem_chen}
For any $0 \leq t_1 \leq t_2$, we have
\[ \EE \| A_{t_2} \|_2^2 \leq \left( \frac{t_2}{t_1} \right)^3 \EE \| A_{t_1} \|_2^2. \]
\end{lemma}
In \cite[Lemma 8]{chen}, Chen proves an analogous inequality  for  $\Vert A_t\Vert_q$ for all $q\geq 3$.
We were not able to deduce Lemma~\ref{lem_chen} from the
analysis in \cite{chen}, which  seems to breaks down for $q<3$.
We prove Lemma \ref{lem_chen} in
Section~\ref{sec_chen}. Combining Lemma~\ref{lem_eldan} with Lemma~\ref{lem_chen} we arrive at the following.
\begin{corollary}\label{lem_at}
Setting $t_1 = (C\kappa_n^2\cdot \log n)^{-1}$ we have
\[
\EE \vert a_t\vert^2  \leq
C_1 n \cdot t \cdot \max \left \{ 1 , \frac{t^3}{t_1^3} \right \}, \quad \forall t > 0,
\]
where $C, C_1 > 0$ are universal constants.
\end{corollary}
\begin{proof}
Recall~\eqref{eq_1049} and observe that $a_0$ is the barycenter of $\mu$,
which is assumed to be $0$. Thus
\[
\EE \vert a_t\vert^2 = \int_0^t  \EE \Vert A_r\Vert_2^2 \, dr .
\]
For $r \in [0,t_1]$ we have $\EE \Vert A_r\Vert_2^2  \leq C n$ according to Lemma~\ref{lem_eldan}. For $r\geq t_1$,
\[
\EE \Vert A_r\Vert_2^2  \leq \EE  \Vert A_{t_1}\Vert_2^2  \cdot \frac{r^3}{t_1^3}
\leq C_1 n \cdot \frac{r^3}{t_1^3}  ,
\]
by Lemma~\ref{lem_chen}. The result follows by integrating these inequalities from $0$ to $t$.
\end{proof}
\begin{remark}
It might be possible and it could be interesting to remove all stochastic processes from the argument,
and perform our analysis, as well as that of Lee-Vempala \cite{LV} and Chen \cite{chen},
using differentiations along the heat semi-group and integrations by parts in place of It\^o calculus.
The advantage of the stochastic point of view, is that the evolution equation (\ref{eq_910})
provides a convenient coupling between the probability measures $\mu_{s_1}$ and $\mu_{s_2}$,
which enables us to analyze the $s$-derivatives of certain integrals with respect to the measure $\mu_s$.
We remark that an alternative coupling to (\ref{eq_910}), which has so far been less useful for our analysis,
 is provided by the
deterministic evolution equation
\begin{equation} d \theta_t = \frac{1}{2} \left( a(t, \theta_t) + \frac{\theta_t}{t} \right) dt,
\label{eq_914}
\end{equation}
which is referred to as the Kim-Milman map in \cite{KP}.
For any $0 < t_1 < t_2$,  if $\theta_{t_1} / t_1$ is a random vector with law $\mu_{s_1}$ for $s_1 = 1/t_1$, and we run
either the evolution (\ref{eq_910}) or else the evolution (\ref{eq_914}) for $t \in [t_1, t_2]$,
then $\theta_{t_2} / t_2$ is a random vector with law $\mu_{s_2}$ for $s_2 = 1/t_2$.
\end{remark}
The next three sections are dedicated to the proofs of
Proposition \ref{prop_1715}, Lemma \ref{lem_chen} and Lemma \ref{lem_eldan}.
Finally in Section \ref{sec_proof} we explain how all pieces fit together and
prove Theorem \ref{thm_2226}. The basic idea is
that Corollary \ref{lem_at} provides a bound for $\EE |a_t|^2 = \| Q_s x \|_{L^2(\mu_s)}$, with $s = 1/t$,
which together with Proposition \ref{prop_1715} and the $H^{-1}$-inequality
of~\cite{ptrf} yields a bound on the thin-shell parameter of $\mu$.
\section{Spectral measures and heat flow}
\label{sec_1700}
In this section we prove Proposition \ref{prop_1715}.
Assume that $\mu$ is a log-concave probability measure with a smooth, positive density $\rho$ in
$\RR^n$. Recall the spectral decomposition~\eqref{eq_spectraldec} of the Laplace operator associated to
$\mu$, and recall that for a given function $f \in L^2(\mu)$ the notation $\nu_f$ stands for its spectral measure.
 Write $H^1(\mu)$ for the space of all functions $f \in L^2(\mu)$ whose weak derivatives $\partial^1 f,\ldots,\partial^n f$ exist and belong to $L^2(\mu)$.
For $f \in H^1(\mu)$ we define
$$ \| f \|_{\dot{H}^{1}(\mu)}^2 = \int_{\RR^n} |\nabla f|^2 d \mu $$
and $\| f \|_{H^1(\mu)} = \sqrt{ \| f \|_{L^2(\mu)}^2 + \| f \|_{\dot{H}^{1}(\mu)}^2}$. It is proven in the appendix of \cite{BK} that compactly-supported smooth functions are dense in $H^1(\mu)$, with respect to the $H^{1}(\mu)$-norm.
It is explained in \cite[Chapter 6]{D} that since $L$ is the Friedrich extension of the operator initially defined on smooth, compactly-supported functions,
for any $f \in H^1(\mu)$,
\begin{equation}  \| f \|_{\dot{H}^{1}(\mu)}^2 = \int_{0}^{\infty} \lambda d \nu_f(\lambda). \label{eq_1829}
\end{equation}
Moreover, $H^1(\mu)$ is the space of all functions in $L^2(\mu)$ for which the integral on the right-hand side of (\ref{eq_1829}) converges.
The next lemma expresses the fact proven in \cite{KP} that the function
$s \mapsto \log \| Q_s g \|_{L^2(\mu_s)}^2$ is convex, and hence it lies above its tangent at zero.
The lemma implies  that $Q_s$ does not reduce the norm of a low-energy function by much.
\begin{lemma} \label{lem_2249}
Let $g \in H^1(\mu)$ satisfy $\int g^2 d \mu =1$.
Denote $E = \int_{\RR^n} |\nabla g|^2 d \mu$.
Then for all $s > 0$,
\begin{equation}  \| Q_s g \|_{L^2(\mu_s)}^2 \geq \exp(-s E).
\label{eq_2249} \end{equation}
\end{lemma}
\begin{proof}  Since compactly-supported smooth functions are dense in $H^1(\mu)$,
by \cite[Lemma 2.5]{KP} it suffices to prove (\ref{eq_2249}) for a compactly-supported, smooth function $g$.
According to \cite[Section 2]{KP}, for any $s > 0$, the $s$-derivative of the function
$-\log \| Q_s g \|_{L^2(\mu_s)}^2$ is the Rayleigh quotient
$$ R_g(s) = \| Q_s g \|_{\dot{H}^1(\mu_s)}^2 / \| Q_s g \|_{L^2(\mu_s)}^2, $$
which is continuous and non-increasing
in $s \in [0, \infty)$, with $E = R_g(0)$. Hence,
$$ \log \| Q_s g \|_{L^2(\mu_s)}^2 = -\int_0^s R_g(x) dx \geq -s R_g(0) = -s E, $$
proving (\ref{eq_2249}).
\end{proof}
Lemma \ref{lem_2249} yields a lower bound for $\langle A g, g \rangle_{L^2(\mu)}$ with $A = P_s Q_s = Q_s^* Q_s$ when $g$ is a low energy function. The next lemma shows that if in addition $g$ is an orthogonal projection of a certain function $f$ then we can upgrade this to a lower bound for $ \langle A f, f \rangle_{L^2(\mu)}$.
\begin{lemma}\label{lem_1605}
Let $A$ be a self-adjoint, positive semi-definite operator in $L^2(\mu)$ whose operator norm is at most one.
 Let $f \in L^2(\mu)$ satisfy $\| f \|_{L^2(\mu)} = 1$.
Let $\eps, \beta \in (0,1]$, and assume that $g \in L^2(\mu)$ satisfies $\langle A g, g \rangle_{L^2(\mu)} \geq (1 - \eps) \beta$
and $\beta = \| g \|_{L^2(\mu)}^2 = \langle f, g \rangle_{L^2(\mu)}$.
Then,
$$ \langle A f, f \rangle_{L^2(\mu)} \geq \frac{\beta}{4} - \eps. $$
\end{lemma}
\begin{proof} The norm and scalar product in this proof are those of $L^2(\mu)$.
Consider the spectral decomposition
$$ A = \int_0^1 \lambda d F_{\lambda} $$
for a certain increasing family of orthogonal projections $(F_{\lambda})_{\lambda \in [0,1]}$.
For any $0 \leq r < 1$,
$$ (1 - \eps) \beta \leq \langle A g, g \rangle \leq r \| F_r g \|^2 + (\| g \|^2 - \| F_r g \|^2) = \beta + (r-1) \| F_r g \|^2. $$
Setting $r = 1 - 4 \eps / \beta$ we obtain
$$\| F_r g \| \leq \sqrt{ \frac{\eps \beta}{1-r}} = \beta / 2. $$
Consider the orthogonal projection $P = \id - F_r$. Since $\| F_r f \| \leq \| f \| = 1$,
\begin{equation}  \beta = \langle f, g \rangle = \langle F_r f, F_r g \rangle +
\langle Pf, P g \rangle \leq \| F_r g \| + \langle Pf, P g \rangle \leq \frac{\beta}{2} + \langle Pf, P g \rangle.
\label{eq_1554} \end{equation}
Since $\| g \|^2 = \beta$ we have $\| P g \| \leq \| g \| \leq \sqrt{\beta}$. By the Cauchy-Schwartz inequality and (\ref{eq_1554}),
$$ \| P f \| \geq \frac{\langle Pf, Pg \rangle}{\|Pg\|} \geq \frac{\beta}{2 \|Pg\|} \geq \frac{\sqrt{\beta}}{2}. $$
Hence,
$$ \langle A f, f \rangle \geq r \| P f \|^2 \geq \frac{\beta r}{4} =
\frac{\beta}{4} - \eps,  $$
which is the conclusion of the lemma.
\end{proof}
\begin{proof}[Proof of Proposition \ref{prop_1715}]
The norm and scalar product in this proof are those of $L^2(\mu)$, unless stated otherwise.
Fix $\lambda \geq 0$ and set $g = E_\lambda f$. Then
$$ \beta := \| g \|^2 =  \langle f,g\rangle = \langle f, E_\lambda f \rangle = \nu_f ( [0,\lambda] ).  $$
Additionally $E_0 g = E_0 f = 0$. Moreover, since $\nu_g$ is supported on $[0,\lambda]$
\begin{equation}
E := \int_{0}^{\infty} x d \nu_g( x)
\leq \lambda \nu_g( [0,\lambda]) = \lambda \nu_f ( [0,\lambda]) = \lambda \beta.
\label{eq_1037} \end{equation}
Hence $g \in H^1(\mu)$ and $E = \int_{\RR^n} |\nabla g|^2 d \mu$.
By Lemma~\ref{lem_2249} and~\eqref{eq_1037},
\begin{equation}
\langle P_s Q_s g, g \rangle = \| Q_s g \|_{L^2(\mu_s)}^2
\geq \beta \exp(-s E / \beta) \geq \beta \exp( -s\lambda) \geq \beta ( 1 - \eps ) , \label{eq_1621}
\end{equation}
for $\eps = s \lambda$.
The operator $Q_s: L^2(\mu) \rightarrow L^2(\mu_s)$ is a contraction with $Q_s(1) = 1$,
according to \cite{KP}. Hence the operator $A = P_s Q_s = Q_s^* Q_s$ is a positive semi-definite,
self-adjoint operator of norm one in $L^2(\mu)$.
By (\ref{eq_1621}) we have $\langle A g, g \rangle \geq (1 - \eps) \beta$.
Lemma \ref{lem_1605} implies that
\[
\Vert Q_s f \Vert_{L^2(\mu_s)}^2 =
\langle A f, f \rangle \geq \frac{\beta}{4} - \eps
= \frac 14 \cdot \langle E_\lambda f , f \rangle - s\lambda ,
\]
which is the desired inequality.
\end{proof}
\begin{remark}\label{rem_improve}
Proposition \ref{prop_1715} would have been improved upon  had we known that
for any $s > 0$,
\begin{equation}  P_s Q_s \geq e^{s L} . \label{eq_1628}
\end{equation}
Currently we do not have a counterexample to (\ref{eq_1628}). In fact, (\ref{eq_1628}) holds true in a weak sense,
since for any continuous, increasing test function $\vphi: [0,1] \rightarrow [0, \infty)$ that vanishes in a neighborhood of zero,
\begin{equation}  \tr \, \vphi(P_s Q_s) \geq \tr \, \vphi(e^{s L}). \label{eq_1641} \end{equation}
Indeed, inequality (\ref{eq_1641}) follows from some spectral theory, combined with the fact that by Lemma \ref{lem_2249},
for any $f \in {\rm Dom}(L) \subseteq H^1(\mu)$ with $\| f \|_{L^2(\mu)} = 1$,
$$ \langle P_s Q_s f, f \rangle_{L^2(\mu)} = \| Q_s f \|_{L^2(\mu_s)}^2
\geq \exp(-s \| f \|_{\dot{H}^1(\mu)}^2) = \exp(s \langle L f, f \rangle_{L^2(\mu)}).
$$
\end{remark}
\section{Controlling the growth of the covariance matrix}\label{sec_chen}
In this section we prove Lemma~\ref{lem_chen}.
Given $t>0$ we say that an absolutely-continuous probability measure $\mu$ on $\RR^n$ (or a random vector $X$ in $\RR^n$) is \emph{$t$-uniformly log-concave}
if it has a density of the form
\[
x \mapsto \exp \left( - \phi(x) - \frac t2 \vert x\vert^2  \right) ,
\]
for some convex function $\phi: \RR^n \rightarrow \RR \cup \{ + \infty \}$.
In other words $\mu$ is more log-concave than the
Gaussian measure with covariance $\frac 1t \id$.

\begin{lemma}\label{lem_1612}
Suppose that $\mu$ is a centered,
$1$-uniformly log-concave probability measure on $\RR^n$ and let $(B_t)$ be a standard Brownian motion in $\RR^n$ with $B_0 = 0$.
There there exists an adapted process $(Q_t)$ of symmetric matrices such that $\int_0^1 Q_t d B_t$ has law $\mu$,
and such that almost surely,
\begin{equation}\label{eq_Qt}
0 \leq Q_t \leq \id , \quad \forall t\in [0,1].
\end{equation}
\end{lemma}

\begin{proof} One possibility is introduce the Brenier map $T: \RR^n \rightarrow \RR^n$ pushing forward the standard Gaussian measure to $\mu$. This map is the gradient of a convex function and it
satisfies that $T(B_1)$ has law $\mu$. See e.g.
Villani's book \cite{villani} for information about the Brenier map. Since $\mu$ is $1$-uniformly log-concave the map $T$ is a contraction, according to the Caffarelli
theorem. By elementary properties of the heat kernel, for any $t > 0$, the smooth map $P_{1-t}(T): \RR^n \rightarrow \RR^n$ is still the gradient of a convex function,
and $$ M_t(x) := \nabla P_{1-t}(T)(x) \in \RR^{n \times n} $$ is a symmetric matrix with $0 \leq M_t(x) \leq \id$ for any $x \in \RR^n$ and any $0 < t < 1$.
Since $(P_{1-t} T)(B_t)$ is a martingale, we have with $Q_t = M_t(B_t)$,
$$ T(B_1) = (P_1 T)(B_0) + \int_0^1 Q_t d B_t = \int_0^1 Q_t d B_t, $$
where $(P_1 T)(0) = \EE T(B_1) = 0$ since $\mu$ is centered. This proves the lemma. Another possibility is to note that it is shown in Eldan and Lehec \cite{EL}
by using stochastic localization
that $\int_0^{\infty} P_t d B_t$ has law $\mu$ for some adapted process $(P_t)$ of symmetric matrices with $0 \leq P_t \leq e^{-t/2}$.
This also implies the required result, via a time change.
\end{proof}
The main step in the proof of  Lemma~\ref{lem_chen} is the following
lemma, whose proof is loosely inspired by the article of Jiang, Lee
and Vempala~\cite{JLV} in which they prove a related inequality,
but with a different set of hypotheses.
\begin{lemma}\label{lem_keyChen}
Let $t>0$, let $\mu$ be a probability measure on $\RR^n$ which is
$t$-uniformly log-concave and centered, and let $X,Y$ be two independent
random vectors with law $\mu$. Then
\[
\EE \langle X,Y\rangle^3 \leq \frac 3t \cdot \EE  \langle X,Y\rangle^2  = \frac 3t \cdot \tr ( A^2 ) ,
\]
where $A$ is the covariance matrix of $\mu$.
\end{lemma}
\begin{remark}
The value of the constant in Lemma \ref{lem_keyChen} directly controls
the exponent of the logarithm in our main result. The value $3$ is the best we could do, but it
is probably not optimal. For instance, the inequality of Lemma \ref{lem_keyChen} holds true with constant $2$ in dimension one.
\end{remark}

\begin{proof}[Proof of Lemma \ref{lem_keyChen}]
By homogeneity it is enough to prove the statement for $t= 1$ (observe that if
$X$ is $t$-uniformly log concave then $\frac 1 {\sqrt t} \cdot X$
is $1$-uniformly log-concave). We apply Lemma \ref{lem_1612} and let $(X_t)$ be the martingale given by
\[
X_t = \int_0^t Q_s dB_s .
\]
Set $X=X_1$ and let $Y$ be a copy of $X$ that is
independent of the Brownian motion $(B_t)$. In particular
$Y$ is independent of $X$. By It\^o's formula
\begin{equation}\label{eq_stepchen}
\EE \langle X,Y\rangle^3 = 3 \int_0^1 \EE \left[ \langle X_t, Y\rangle \cdot \vert Q_t Y\vert^2 \right]  \, dt  .
\end{equation}
Fix $x\in\RR^n$ and a symmetric matrix $Q$. Recall that $A$ is the covariance matrix of $Y$. Since $Y$ is centered we have
\[
\EE \left[ \langle x,Y\rangle \cdot \vert Q Y\vert^2 \right]
\leq \EE \left[ \langle x , Y \rangle^2 \right]^{1/2} \cdot \var ( \vert Q Y \vert^2  )^{1/2}
= \langle Ax,x\rangle^{1/2} \cdot \var ( \vert Q Y \vert^2  )^{1/2} .
\]
Let us apply the Poincar\'e inequality in order to bound the variance of $\vert QY\vert^2$.
It is well known that a $1$-uniformly log-concave random vector satisfies the Poincar\'e inequality with constant $1$. However, according to
Cordero-Erausquin, Fradelizi and Maurey~\cite[Lemma 2]{CEFM}, under the $1$-uniform log-concavity assumption,
functions whose gradient is centered satisfy the Poincar\'e inequality with constant $\frac 12$ (rather than $1$).
The gradient of $\vert QY\vert^2$ is $2Q^2Y$, which is indeed centered, and we obtain
\[
\var ( \vert QY\vert^2 ) \leq \frac 12 \cdot \EE \vert 2 Q^2 Y \vert^2  = 2 \cdot \tr ( Q^4 A ) .
\]
Plugging this back in~\eqref{eq_stepchen}, using the Cauchy-Schwarz inequality, and pulling the power $1/2$
outside of the integral on $[0,1]$, we get
\begin{equation}\label{eq_stepchen2}
\EE [ \langle X,Y\rangle^3 ] \leq
3\sqrt 2 \cdot \left( \int_0^1 \EE\langle A X_t, X_t\rangle \cdot \EE \tr(Q_t^4 A) \, dt \right)^{1/2}.
\end{equation}
For $t\in[0,1]$ set $M_t=\EE Q_t^2$.
By It\^o's formula the covariance matrix of $X_t$ is $\int_0^t M_s \, ds $. In particular $\int_0^1 M_s \, ds$
is the covariance matrix of $X$, namely $A$. Additionally, by~\eqref{eq_Qt} we have $Q_t^4 \leq Q_t^2$ almost surely,
hence $\EE \tr( Q_t^4 A ) \leq \tr ( M_t A )$, since $A$ is a positive semi-definite matrix. Thus
\begin{align} \nonumber
\int_0^1 \EE\langle A X_t, X_t\rangle \cdot \EE \tr(Q_t^4 A) \, dt
& \leq \int_0^1 \left( \int_0^t \tr(M_s A) \, ds \right) \cdot \tr ( M_t A ) \, dt  \\
& = \frac 12 \left( \int_0^1 \tr(M_t A) \, dt \right)^2 = \frac 12 \tr ( A^2 )^2.
\label{eq_1657} \end{align}
The result follows from ~\eqref{eq_stepchen2} and \eqref{eq_1657}.
\end{proof}
We can now prove the Chen type bound for $q=2$.
\begin{proof}[Proof of Lemma~\ref{lem_chen}]
The matrix $A_t$ satisfies
\[
d A_t = \left\langle  \int_{\RR^n} (x-a_t)^{\otimes 3} p_t (x) \, dx , d B_t \right\rangle - A_t^2 \, dt.
\]
See for instance~\cite{LV} where this computation is explained in detail.
It\^o's formula
 yields
\begin{equation}\label{eq_dat2}
d \tr ( A_t^2 ) = 2 \sum_{i,j=1}^n A_{ij,t} \langle H_{ij,t} , d B_t \rangle  + \sum_{i,j=1}^n \vert H_{ij,t} \vert^2 \, dt - 2 \tr (A_t^3) \, dt ,
\end{equation}
where $A_t = (A_{ij,t})_{i,j=1,\ldots,n} \in \RR^{n \times n}$  and
\[
H_{ij,t} = \int_{\RR^n} (x-a_t)_i (x-a_t)_j (x-a_t) p_t(x) \, dx .
\]
We claim that with probability one
\begin{equation}
\sum_{i,j=1}^n \vert H_{ij,t} \vert^2  \leq \frac{3}{t} \cdot \tr(A_t^2).
\label{eq_1632}
\end{equation}
Indeed, conditioning on $(B_t)$, we
observe that the term on the left-hand side
of~\eqref{eq_1632} is precisely $\EE \langle X,Y\rangle^3$ where $X,Y$ are two independent
random vectors having density $x\mapsto p_t(x+a_t)$. Since the density $p_t$ is $t$-uniformly log-concave, the random
vectors $X,Y$ are $t$-uniformly log-concave, and they are centered by the definition of $a_t$.
Hence (\ref{eq_1632}) follows from Lemma~\ref{lem_keyChen}.
Consequently, by taking expectation in~\eqref{eq_dat2}
we obtain
\[
\frac d{dt} \EE \tr (A_t^2) \leq \frac 3t \cdot \EE \tr(A_t^2) - 2 \cdot \EE \tr (A_t^3) \leq \frac 3t \cdot \EE \tr(A_t^2) .
\]
Integrating this differential inequality yields the conclusion of the lemma.
\end{proof}

\begin{remark} By using the inequality
of Cordero-Erausquin, Fradelizi and Maurey~ \cite{CEFM}
one can improve the exponent in  Chen's lemma \cite[Lemma 8]{chen}
from $2q$ to $q-1$, but still only for $q \geq 3$.
With this improved exponent, it is possible
to adapt our proof of Theorem \ref{thm_2226}
and replace the usage of Lemma \ref{lem_chen}
by the case $q = 3$ of Lemma 8 from \cite{chen}.
\end{remark}
\section{Bounding the covariance matrix for small times}\label{sec_eldan}
In this section we prove Lemma~\ref{lem_1135}.
For a centered, log-concave random vector $X$ in $\RR^n$
we set
\[
\kappa_X^2 = \sup_{\theta\in S^{n-1}} \left\{
 \left\Vert \EE \langle X,\theta\rangle (X\otimes X) \right\Vert_2^2 \right\} ,
\]
and recall that $\kappa_n^2$ is the
supremum of $\kappa^2_X$ over all isotropic log-concave
random vectors $X$
in $\RR^n$. The following lemma is elementary.
\begin{lemma}\label{lem_kappa}
If $X$ is a centered, log-concave random vector in $\RR^n$ then
\[
\kappa_X^2 \leq \Vert \cov(X) \Vert^3_{op} \cdot \kappa_n^2 ,
\]
where $\cov ( X)$ is the covariance matrix of $X$, and the
norm here is the operator norm.
\end{lemma}
\begin{proof}
Both terms of the inequality are homogeneous of degree $6$
in $X$ so we can assume that $\cov(X)$ has operator norm $1$. Therefore $\cov ( X ) \leq \id$, and in particular $\id-\cov(X)$
is a positive semi-definite matrix. Let $Y$ be a log-concave
random vector independent of $X$, having covariance $\id - \cov(X)$, such that $Y$ coincides in law with $-Y$, i.e. $Y$ is symmetric. For instance $Y$ could be a centered Gaussian random vector with the appropriate
covariance matrix. Under these conditions
\[
\EE [ (X+Y)^{\otimes 3}  ] = \EE [ X^{\otimes 3} ] .
\]
Indeed, all of the other terms in the expansion of the
tensor $(X+Y)^{\otimes 3}$ have zero expectation.
Thus $\kappa^2_X = \kappa^2_{X+Y}$ and the latter quantity is at most $\kappa_n^2$, since $X+Y$ is log-concave and isotropic.
\end{proof}
Recall that $\mu$ is an isotropic, log-concave probability
measure in $\RR^n$ with density $p_0$,
and that $(p_t)$ is the associated stochastic localization process. As above we denote by $(a_t)$ and $(A_t)$
the corresponding processes of barycenters and covariance matrices, respectively.
\begin{lemma}\label{prop_machin}
For every $t\leq (C \kappa_n^2 \cdot \log n)^{-1}$
we have
\[
\PP ( \Vert A_t\Vert_{op} \geq 2 ) \leq \exp ( - (Ct)^{-1} ),
\]
where $C > 0$ is a universal constant.
\end{lemma}
\begin{proof} Recall that the matrix process  $(A_t)_{t \geq 0}$ satisfies
\[
d A_t = \left\langle  \int_{\RR^n} (x-a_t)^{\otimes 3} p_t (x) \, dx , d B_t \right\rangle - A_t^2 \, dt.
\]
We will apply It\^o's formula to a smooth approximation of
$\Vert A_t \Vert_{op}$. Let $\beta > 0$ be a parameter to be determined soon. Set
\[
\Phi_t = \frac 1\beta \log \tr \left(  e^{\beta A_t}  \right)  ,
\]
and note that
\[
\Vert A_t \Vert_{op} \leq \Phi_t \leq \Vert A_t\Vert_{op} + \frac{ \log n } \beta .
\]
Write  $0 \leq \lambda_{1,t} \leq \dotsb \leq \lambda_{n,t}$
for the eigenvalues of $A_t$, repeated according to
their multiplicity, and let $e_{1,t},\ldots,e_{n,t} \in \RR^n$ be a corresponding
orthonormal basis of eigenvectors. For easing  the computation of $d \Phi_t$, we first consider  the case where
the eigenvalues $\lambda_{1,t} \leq \dotsb \leq \lambda_{n,t}$
are almost surely distinct for all positive time. Then It\^o's formula gives (see e.g. \cite{Eldan2})
\[
d \lambda_{i,t} =  \langle u_{ii,t} , d B_t  \rangle - \lambda_{i,t}^2 \, dt
+ \sum_{j\colon j\neq i} \frac{ \vert u_{ij,t} \vert^2 }{\lambda_{i,t}-\lambda_{j,t} } \, dt ,
\]
where
\[
u_{ij,t} = \int_{\RR^n} \langle x-a_t , e_{i,t}\rangle\langle x-a_t, e_{j,t} \rangle  (x-a_t) p_t (x) \, dx.
\]
(The choice of the orthonormal basis of eigenvectors, which are determined only up to a sign, does not affect
the above expression for the It\^o derivative of $\lambda_{i,t}$).
Next we apply the It\^o formula  to the smooth function
\[
f(\lambda_1,\ldots,\lambda_n) = \frac 1\beta \log \left( \sum_{i=1}^n e^{\beta \lambda_i} \right)
\]
and obtain
\begin{align}
\nonumber
d \Phi_t & = \sum_i \alpha_i \langle u_{ii} , d B_t \rangle - \sum_i \alpha_i \lambda_i^2  dt + \sum_{i\neq j}
\frac{ \alpha_i \vert u_{ij} \vert^2 }{\lambda_{i}-\lambda_{j} } \, dt \\
& + \frac \beta 2 \sum_{i} \alpha_i \vert u_{ii} \vert^2 \, dt
- \frac \beta 2 \left\vert \sum_{i} \alpha_i u_{ii}\right\vert^2 \, dt  ,
\label{eq_step_phi}
\end{align}
where we dropped the dependence in $t$ to lighten notations and where
\[
\alpha_i = \frac{ e^{\beta\lambda_i} } { \sum_{j=1}^n e^{\beta \lambda_j} } .
\]
Symmetrizing the third term of~\eqref{eq_step_phi},
namely replacing $\alpha_i$ by $\frac12 (\alpha_i - \alpha_j)$, we
observe that the expression for $d \Phi_t$ still makes sense by continuity when the eigenvalues
are not necessarily distinct. This suggests that this
expression for $d\Phi_t$ remains valid if we drop this assumption.
As a matter of fact, applying It\^o's formula directly to the
function $A\mapsto \frac 1\beta \log \tr ( e^{\beta A}  )$ does lead to the
same expression for $d\Phi_t$.
Using the inequality
\[
\frac{ e^x - e^y }{ x -y } \leq \frac { e^x + e^y } 2
\]
 to upper bound the third term of~\eqref{eq_step_phi}, and dropping the
non-positive terms, we arrive at the inequality
$$
d \Phi_t \leq  \sum_i \alpha_i  \langle u_{ii} , d B_t \rangle
+ \frac \beta2 \sum_{i,j} \alpha_i \vert u_{ij} \vert^2 dt.
$$
By the Cauchy-Schwartz inequality and
the  reverse H\"older inequalities
for log-concave measures (e.g. \cite[Theorem 2.4.6]{BGVV}),
\begin{align*}  |u_{ii}| & \leq \sup_{\theta \in S^{n-1}} \int_{\RR^n}
\langle x - a_t, e_i \rangle^2 |\langle x - a_t, \theta \rangle| p_t(x) dx \\ & \leq \left( \int_{\RR^n}
\langle x - a_t, e_i \rangle^4 p_t(x) dx \right)^{1/2} \cdot \| A_t \|_{op}^{1/2}
\leq C \| A_t \|_{op}^{3/2}, \end{align*}
for some universal constant $C > 0$. According to  Lemma~\ref{lem_kappa}, for any fixed $i$,
\[
\sum_{j} \vert u_{ij} \vert^2 \leq \Vert A_t\Vert_{op}^3 \cdot \kappa_n^2.
\]
Since the $\alpha_i$'s are positive and add up to $1$, we finally obtain
\[
d \Phi_t = \langle v_t , d B_t \rangle + c_t \, dt ,
\]
where
\begin{equation}\label{eq_vt}
\vert v_t \vert^2 \leq C \cdot \Vert A_t \Vert_{op}^3  \quad \text{and} \quad
c_t \leq \frac 12 \beta \cdot \Vert A_t\Vert_{op}^3 \cdot \kappa_n^2 .
\end{equation}
Let $\tau$ be the following stopping time:
\[
\tau = \inf \{ t\geq 0 \, ; \,  \Vert A_t \Vert_{op} \geq 2 \}.
\]
Choose $\beta = 2 \log n$ and suppose that $t\leq (32 \cdot \kappa_n^2 \cdot \log n)^{-1}$.
Note that $\Phi_0 = 3/2$ since $\mu$ is isotropic.
Because of~\eqref{eq_vt} and by the definition of $\tau$
we have
\[
\Phi_{t\wedge \tau} \leq \Phi_0 + M_{t } + 8 t \cdot \kappa_n^2 \cdot \log n \leq \frac 32 + M_{t} + \frac 14 ,
\]
where $t \wedge \tau = \min \{ t, \tau \}$ and $(M_t)$ is the martingale
\[
M_t = \int_0^{t\wedge\tau} \langle v_s , dB_s\rangle .
\]
If $\Vert A_t\Vert_{op} \geq 2$ then $\tau\leq t$, $\Vert A_{t\wedge\tau} \Vert_{op} \geq 2$
and also $\Phi_{t\wedge\tau} \geq 2$.
In view of the preceding inequality this implies that
\[
\PP( \Vert A_t\Vert_{op} \geq 2 ) \leq \PP \left( M_t \geq \frac 14 \right).
\]
The martingale $(M_s)$ satisfies $M_0 = 0$ and by~\eqref{eq_vt}  its quadratic variation at time $t$ satisfies
\[
[M]_t = \int_0^{t\wedge \tau} \vert v_s\vert^2 \, ds \leq C' t , \quad \text{almost surely}.
\]
The martingale lemma spelled out below thus implies
\[
\PP \left( M_t \geq \frac 14 \right) \leq e^{-(C''t)^{-1}}  ,
\]
which concludes the proof.
\end{proof}
The following deviation inequality for
martingales with bounded quadratic variation is folklore.
We provide a short proof for completeness.
\begin{lemma}
Let $(M_t)_{t \geq 0}$ be a continuous martingale
satisfying $M_0 = 0$ and $[M]_t \leq \sigma^2$ almost surely for some fixed time $t>0$ and some constant $\sigma > 0$. Then
\[
\PP( M_t \geq u ) \leq e^{ - u^2 / 2 \sigma^2 } , \quad \forall u \geq 0 .
\]
\end{lemma}
\begin{proof} Let $\lambda >0$. By It\^o's formula the process $(D_s)$ given by
\[
D_s = \exp\left( \lambda M_s - \frac{\lambda^2} 2 [M]_s \right)
\]
is a positive local martingale, hence a super-martingale by Fatou's lemma.
In particular $\EE [ D_t ] \leq \EE [ D_0 ] = 1$. In view of the hypothesis, this yields
\[
\EE [ \exp( \lambda M_t ) ] \leq \exp ( \lambda^2 \sigma^2 / 2 ) .
\]
Now apply the Markov inequality and optimize in $\lambda$.
\end{proof}
\begin{corollary} \label{cor_truc}
Let $t\leq (C \kappa_n^2 \cdot \log n)^{-1}$ and let $p\geq 1$. Then
\[
\EE [\Vert A_t \Vert_{op}^p ] \leq C_p ,
\]
where $C > 0$ is a universal constant and $C_p > 0$ is a constant depending only on $p$.
\end{corollary}
\begin{proof}
Since $A_t$ is the covariance matrix of a measure which is
more log-concave than the Gaussian measure with covariance $\frac 1t \id$
we have $\Vert A_t \Vert_{op} \leq \frac 1t$, almost surely. Applying Lemma~\ref{prop_machin}
we thus get for $t\leq (C \kappa_n^2 \cdot \log n)^{-1}$,
\[
\EE[ \Vert A_t\Vert_{op}^p ] \leq 2^p + \frac 1{t^p} \PP ( \Vert A_t\Vert_{op} \geq 2 )
\leq 2^p + \frac 1{t^p} e^{- 1/(Ct) } \leq 2^p + C^p p! ,
\]
and the corollary is proven.
\end{proof}
Clearly $\Vert A_t \Vert_p^p \leq n \Vert A_t\Vert_{op}^p$, and hence  Corollary~\ref{cor_truc} yields
the desired Lemma~\ref{lem_1135}.
\begin{remark}
Corollary~\ref{cor_truc} for $p=1$ recovers Eldan's theorem.
Indeed it implies that setting $t_1 :=(C \kappa_n^2\cdot \log n )^{-1}$, we have
\[
\EE \left[ \int_0^{t_1} \Vert A_t\Vert_{op} \, dt \right] \leq C' t_1 \leq \hat C \cdot \frac1{\log n} = o(1) .
\]
This is well known to imply $\psi_\mu^2 \leq \tilde C \cdot t_1^{-1} = C_1 \kappa_n^2 \cdot \log n$,
see for instance~\cite[page 9]{k_chen}. We thus get
\[
\psi_n^2 \leq  C_1 \kappa_n^2 \cdot \log n ,
\]
which is (the hard part of) Eldan's inequality~\eqref{eq_eldanbis}.
\end{remark}

\section{Proof of the main result}
\label{sec_proof}
In this section we prove Theorem \ref{thm_2226}.
Let $\mu$ be an isotropic, log-concave probability measure in $\RR^n$ with
a smooth positive density. Let us  furthermore assume that
\begin{equation} \sigma_\mu > \sigma_n / 2. \label{eq_1139} \end{equation}
Thus the thin-shell parameter of $\mu$ is nearly as large as possible. The requirement (\ref{eq_1139})
is consistent with the assumption that $\mu$ has a smooth, positive density. Indeed by convolving $\mu$
with a tiny Gaussian and normalizing back to isotropicity, we obtain a smooth positive density,
and the change in $\sigma_{\mu}$ can be made arbitrarily small.
For $f \in L^2(\mu)$ with $\int f d \mu = 0$
write
\begin{equation}  \| f \|_{H^{-1}(\mu)} = \sup \left \{ \int_{\RR^n} f u d \mu \, ; \, u \in L^2(\mu) \textrm{ is locally-Lipschitz with } \int_{\RR^n} |\nabla u|^2 d \mu \leq 1 \right \}. \label{eq_1900} \end{equation}
Recall that for $f \in L^2(\mu)$ we write 
$\nu_f$ for the spectral measure of $f$ relative to the Laplace operator
associated with $\mu$. According to (\ref{eq_1851}), setting 
$$ \lambda_1 = \frac{1}{C_P(\mu)} $$
we have that $\nu_f([0, \lambda_1)) = 0$ for 
any $f \in L^2(\mu)$ with $\int f d \mu = 0$ (i.e., $f$ is centered).
From (\ref{eq_1829})  and (\ref{eq_1900}) we can deduce that when $f \in L^2(\mu)$ is centered, 
\begin{equation}
\| f \|_{H^{-1}(\mu)}^2 = \int_{0}^{\infty} \frac{d \nu_f(\lambda)}{\lambda}
= \int_{\lambda_1}^{\infty} \frac{d \nu_f(\lambda)}{\lambda}.
\label{eq_1802} \end{equation}
The following inequality was proven in the case of the uniform measure on a convex body in \cite{ptrf} and for a general log-concave measure
in Barthe and Klartag \cite{BK}. According to Proposition 10 in \cite{BK}, for any smooth function $f: \RR^n \rightarrow \RR^n$ with $f, \nabla f \in L^2(\mu)$ such that $\int_{\RR^n} \partial^i f d \mu = 0$ for all $i$,
\begin{equation}
\var_{\mu}(f) \leq \sum_{i=1}^n \| \partial^i f \|_{H^{-1}(\mu)}^2. \label{eq_1856}
\end{equation}
Specializing to the case where $f(x) = |x|^2$ in (\ref{eq_1856}),
which is the main case used already in \cite{ptrf}, we obtain
$$  n \sigma_{\mu}^2 = \var_{\mu}(|x|^2)
\leq 4  \sum_{i=1}^n \| x_i \|_{H^{-1}(\mu)}^2. $$
Since $\mu$ is centered, we may use ~\eqref{eq_1802} and rewrite this as
\begin{equation}\label{eq_sigmaF}
\sigma_\mu^2 \leq \frac{4}{n} \sum_{i=1}^n \int_{\lambda_1}^\infty \frac{d \nu_{x_i}(\lambda)}{\lambda}
= 4 \int_{\lambda_1}^{\infty} \frac{F(\lambda)}{\lambda^2} d \lambda,
\end{equation}
where the last equality follows by integration by parts, and 
\[
F(\lambda ): = \frac 1n \sum_{i=1}^n \nu_{x_i} ( [0,\lambda] )
\in [0,1]
\]
is the average spectral mass of the coordinate functions below level $\lambda$. 
We are now in a position to prove our main result.
\begin{proof}[Proof of Theorem~\ref{thm_2226}]
Applying Proposition~\ref{prop_1715} with the coordinate function $x_i$ and
summing over $i$ yields
\begin{equation}\label{eq_Flambda2}
F(\lambda ) \leq C \left(  \frac 1n \cdot \Vert Q_s x\Vert^2_{L^2 ( \mu_s ) } + \lambda s \right) ,
\end{equation}
for every positive $\lambda$ and $s$.
Let us translate this to the normalization $t = 1/s$ of Eldan's stochastic localization.
By (\ref{eq_1044}) we can rewrite~\eqref{eq_Flambda2} as
\begin{equation}\label{eq_Flambda}
F(\lambda) \leq C  \left(  \frac 1n \cdot \EE \vert a_t \vert^2  + \frac \lambda t \right).
\end{equation}
According to Corollary~\ref{lem_at}, for all $t>0$ we have
\begin{equation}\label{eq_steptruc}
\EE \vert a_t\vert^2 \leq C_1 t \cdot \max \left\{ 1,  (t/t_1)^3  \right\} \cdot n ,
\end{equation}
where $t_1 = (C\kappa_n^2\cdot \log n )^{-1}$. From 
(\ref{eq_536}), (\ref{eq_1210}) and Eldan's Theorem 
in the form of inequality (\ref{eq_eldanbis}) above,
\begin{equation} \label{eq_1008}
 t_1 = \frac { c}{\kappa_n^2 \log n} \leq \frac {c'} {\psi_n^2} \leq \frac {\tilde{c}} { C_P(\mu) } =  \tilde{c} \cdot \lambda_1.
\end{equation}
Given $\lambda \geq \lambda_1$,  combining ~\eqref{eq_Flambda} and ~\eqref{eq_steptruc} and choosing $t = \lambda^{1/5} \cdot t_1^{3/5} \geq c \cdot t_1$
yields
\begin{equation}\label{eq_12345}
F(\lambda) \leq C t^4 \cdot  t_1^{-3} +  C' \lambda \cdot t^{-1} \leq \tilde{C} \lambda^{4/5} \cdot t_1^{-3/5} .
\end{equation}
From ~\eqref{eq_sigmaF}, (\ref{eq_1008}) and (\ref{eq_12345}), 
\begin{equation}
\sigma_\mu^2 \leq 4 \int_{\lambda_1}^{\infty} \frac{F(\lambda)}{\lambda^2} d \lambda \leq C t_1^{-3/5} \cdot \int_{\lambda_1}^\infty \lambda^{-6/5} \, d\lambda
= C' t_1^{-3/5} \cdot \lambda_1^{-1/5} \leq C' t_1^{-4/5} .
\label{eq_1009} \end{equation}
Now recall the definition of $t_1$, the fact that $\mu$ has a nearly maximal thin-shell constant and
(the easy part of) Eldan's inequality~\eqref{eq_eldanbis}. We thus obtain from (\ref{eq_1139}), (\ref{eq_1008}) and (\ref{eq_1009}) that
\[
\sigma_n^2 \leq C \sigma_\mu^2 \leq C' \left( \kappa_n^2 \log n \right)^{4/5}
\leq \tilde{C} \left( \sigma_n^2 \log^2 n \right)^{4/5} .
\]
Therefore $\sigma_n \leq C (\log n)^4$ and thus $\psi_n \leq C' (\log n)^5$ by one last
application of Eldan's theorem.
\end{proof}
\begin{remark}
The exponent of the logarithmic factor in Theorem \ref{thm_2226} does not seem optimal.
As we already mentioned, one way to decrease it could be to improve the
constant in Lemma~\ref{lem_chen}. Another option is to either lower the gap between $\kappa_n$
and $\sigma_n$ or to replace $\kappa_n$ in Lemma~\ref{lem_eldan} by something smaller,
maybe an averaged version of $\kappa_n$. A matrix-valued version of Lemma \ref{lem_chen} could be useful too.
Lastly, a stronger
version of Proposition~\ref{prop_1715} could hold true, see Remark~\ref{rem_improve}, and
this  would improve our main result.
\end{remark}
\begin{remark} From~\eqref{eq_1802} one can obtain the following alternative expression
for the $H^{-1}$-norm of a centered function $f \in L^2(\mu)$:
\[
\Vert f \Vert_{H^{-1}(\mu)}^2
 = \int_0^\infty \langle e^{sL} f, f  \rangle_{L^2(\mu)} \, ds .
\]
This shows that the
$H^{-1}$-inequality~\eqref{eq_1856} can
be reformulated as follows
\begin{equation}\label{eq_H1again}
\var_\mu (f) \leq \int_0^\infty \langle e^{sL} \nabla f , \nabla f \rangle_{L^2(\mu)} \, ds
= 2  \int_0^\infty \Vert e^{sL} \nabla f \Vert_{L^2(\mu)}^2  \, ds.
\end{equation}
Here the semi-group $e^{sL}$ is applied coordinate-wise to the vector field
$\nabla f$, and the second equality follows from the change of variable $s \to 2s$ and
the fact that $e^{sL}$ is self-adjoint in $L^2 ( \mu )$. This should be
compared with the expression one gets by differentiating the variance of
$f$ along the semi-group $e^{sL}$, namely
\begin{equation}\label{eq_BEvariance}
\var_\mu ( f ) = 2 \int_0^\infty \Vert \nabla (e^{sL} f) \Vert^2_{L^2 (\mu) } \, ds ,
\end{equation}
see for instance~\cite[section 4.2]{BGL}.
Moreover, in the log-concave case,
the Bakry-\'Emery machinery yields the commutation rule
\[
\vert \nabla (e^{sL} f) \vert^2 \leq e^{sL} \vert \nabla f \vert^2 ,
\]
pointwise (in both space and time), see~\cite[Theorem 3.2.3]{BGL}.
However, it does {\it not} imply that
$$
\vert \nabla (e^{sL} f) \vert^2 \leq \vert e^{sL} \nabla f \vert^2
$$
and we suspect that this inequality is not true in general. Nevertheless, after integrating 
in both space and time, this becomes a valid inequality. Indeed, the $H^{-1}$-inequality~\eqref{eq_H1again} and the equality~\eqref{eq_BEvariance} show that when $\mu$ is log-concave, for any $f \in L^2(\mu)$,
\begin{equation}  \int_0^\infty \Vert \nabla (e^{sL} f) \Vert^2_{L^2 (\mu) } \, ds 
\leq \int_0^\infty \Vert e^{sL} \nabla f \Vert_{L^2(\mu)}^2  \, ds. \label{eq_1902} \end{equation}
\end{remark}
%

\end{document}